\documentclass[12pt]{article}
\usepackage{amsmath,amsthm,amsfonts,latexsym,amsopn,verbatim,amscd,amssymb}
\usepackage{pgfplots}
\usepackage{caption}
\theoremstyle{plain}
\newtheorem{theorem}{Theorem}[section]
\newtheorem{lemma}[theorem]{Lemma}
\newtheorem{corollary}[theorem]{Corollary}
\newtheorem{observation}[theorem]{Observation}

\newtheorem{proposition}[theorem]{Proposition}

\newcommand{\bnum}{\begin{enumerate}}
\newcommand{\enum}{\end{enumerate}}

\numberwithin{equation}{section}

\DeclareMathOperator{\Irr}{Irr}

\begin{document}
\title{\textbf{Relative $g$-noncommuting graph of  finite groups}}
\author{Monalisha Sharma and Rajat Kanti Nath\footnote{Corresponding author}
}
\date{}
\maketitle
\begin{center}\small{\it
Department of Mathematical Sciences, Tezpur University,\\ Napaam-784028, Sonitpur, Assam, India.\\


Emails:\, monalishasharma2013@gmail.com and rajatkantinath@yahoo.com}
\end{center}
\begin{abstract}
Let $G$ be a finite group. For a fixed element $g$ in $G$ and a given subgroup $H$ of $G$, the relative $g$-noncommuting graph of $G$ is a simple undirected graph whose vertex set is $G$ and two vertices $x$ and $y$ are adjacent if $x \in H$ or $y \in H$ and  $[x,y] \neq g, g^{-1}$. We denote this graph by $\Gamma_{H, G}^g$. In this paper, we obtain computing formulae  for degree of any vertex in $\Gamma_{H, G}^g$ and characterize whether $\Gamma_{H, G}^g$ is a tree, star graph, lollipop or a complete graph together with some  properties of $\Gamma_{H, G}^g$ involving isomorphism of graphs. 
We also present certain relations between the number of edges in $\Gamma_{H, G}^g$ and certain generalized commuting probabilities of $G$ which give some computing formulae for the number of edges in $\Gamma_{H, G}^g$. 
Finally, we conclude this paper by deriving some bounds for the number of edges in $\Gamma_{H, G}^g$.
\end{abstract}


\section{Introduction}
Throughout the paper, $G$ is a finite non-abelian group and $Z(G) = \{z \in G: zx=xz, \forall x \in G\}$ is the center of $G$. For any subgroup $H$ of $G$, we write $Z(H, G) = \{x \in H : xy = yx, \forall y \in G\}$ and $Z(G, H) = \{x \in G: xy = yx \text{ for all } y \in H\}$, which implies $Z(G,G)=Z(G)$. For any element $x \in G$, we write $C_H(x) = \{y \in H : xy=yx\}$. Clearly, $Z(H,G)=\underset{x \in G}{\bigcap} C_H(x)$. We write $K(H, G) = \{[x,y] : x \in H \text{ and } y \in G\}$, where $[x,y] = x^{-1}y^{-1}xy$, and $[H, G] = \langle K(H, G)\rangle$. Therefore, $[G, G] = G'$ is the commutator subgroup of $G$.

The non-commuting graph of $G$, denoted by $\Gamma_G$, is a simple undirected graph with $G \setminus Z(G)$ as the vertex set and two distinct vertices $x$ and $y$ are adjacent whenever $[x, y] \neq 1$. This graph was originated due to the work of Erd$\ddot{\rm o}$s and Neumann \cite{EN76} in 1976. After that different mathematicians studied different aspects of  $\Gamma_G$ 
(see \cite{AAM06,AF15,AI12, Daraf09, DBBM10,dn-JLTA2018,ddn2018,JDSO2015, JDS2015,JMS2019,JSO2015,Mogh05,MSZZ05,  T08,VK18}),
including characterization of finite non-abelian groups. Various interesting generalizations of $\Gamma_G$ due to Erfanian and his collaborators can be found in  \cite{BEKN14,GET14,TE13, TEJ14}.
In particular, in the year 2013, Tolue and  Erfanian \cite{TE13} 
introduced relative non-commuting graph for a given subgroup $H$ of $G$ which is a simple undirected graph, denoted by $\Gamma_{H, G}$, whose vertex set is $G \setminus Z(H, G)$ and two distinct vertices $x$ and $y$ are adjacent if $x \in H$ or $y \in H$ and $[x, y] \neq 1$. In the year 2014, Tolue, Erfanian and Jafarzadeh  \cite{TEJ14} introduced $g$-noncommuting graph for a given element $g$ of a finite group $G$ which is denoted by $\Gamma_G^g$. Recall that $g$-noncommuting graph of a finite group $G$ is a simple undirected graph whose vertex set is $G$ and two distinct vertices $x$ and $y$ are adjacent if  $[x,y] \neq g$ and $g^{-1}$. 
Fusing the concepts of $\Gamma_{H, G}$ and $\Gamma_G^g$,
in this paper,
we introduce relative $g$-noncommuting graph of a finite group $G$. For a given subgroup $H$ of $G$ and an element $g \in G$, the relative $g$-noncommuting graph of $G$, denoted by  $\Gamma_{H,G}^g$, is defined as the simple undirected graph whose vertex set is $G$ and two distinct vertices $x$ and $y$ are adjacent if  $x \in H$ or $y \in H$ and $[x,y] \neq g$ and $g^{-1}$. Note that if $g = 1$ then the induced subgraph of $\Gamma_{H,G}^g$ on $G \setminus Z(H,G)$ is the relative non-commuting graph for a given subgroup $H$ of $G$, that is $\Gamma_{H,G}$.
Also, if $H = G$ then $\Gamma_{G,G}^g = \Gamma_G^g$. The ring theoretic analogues of $\Gamma_{H,G}$, $\Gamma_G^g$ and $\Gamma_{H,G}^g$ can be found in \cite{dbn2018,Sdn2019} and \cite{Sn20-1} respectively.

Let ${\mathcal{G}_1} + {\mathcal{G}_2}$ be the join of the graphs ${\mathcal{G}_1}$  and ${\mathcal{G}_2}$ and let $\overline{{\mathcal{G}}}$ be the complement of ${\mathcal{G}}$. Then we have the following observations, where $K_n$ is the  complete graph on $n$ vertices. 
\begin{observation}\label{deg_obs_1}
Let $H$ be a subgroup of a finite group $G$ and $g\in G$. 
\begin{enumerate}
\item If $g \notin K(H,G)$ then $\Gamma_{H, G}^g = K_{|H|} + \overline{K_{|G|-|H|}}$ and so
\[\deg(x) = \begin{cases}
|G|-1, & \mbox{if $x \in H$}\\
|H|, & \mbox{$x \in G \setminus H$.}
\end{cases}\]
\item If $K(H,G) = \{1\}$ and $g = 1$ then $\Gamma_{H, G}^g =  \overline{K_{|G|}}$.
\end{enumerate}
\end{observation}
\begin{observation}\label{deg_obs_2}
Let $H$ be a subgroup of a finite group $G$ and $g\in G \setminus K(H,G)$. Then
\begin{enumerate}
\item $\Gamma_{H, G}^g$ is a tree if and only if $H = \{1\}$ and $|H| = |G| = 2$.
\item $\Gamma_{H, G}^g$ is a star graph if and only if $H = \{1\}$.
\item $\Gamma_{H, G}^g$ is a complete graph if and only if $H = G$.
\end{enumerate}
\end{observation}
\noindent Note that if $G$ is abelian or $H=Z(H,G)$ then $K(H,G) = \{1\}$. Therefore, in view of Observation \ref{deg_obs_1}, we shall consider $G$ to be non-abelian, $H$ to be a subgroup of $G$ such that $H \ne Z(H,G)$ and $g \in K(H,G)$ throughout this paper.

 In Section 2, we obtain computing formulae  for degree of any vertex in $\Gamma_{H, G}^g$ and characterize whether $\Gamma_{H, G}^g$ is a tree, star graph, lollipop or a complete graph together with some  properties of $\Gamma_{H, G}^g$ involving isomorphism of graphs. In Section 3, 
we  obtain the number of edges in $\Gamma_{H, G}^g$ using ${\Pr}_g(H, G)$, which is the probability (introduced and studied in \cite{das-nath-10,nath_yadav}) that the commutator of a randomly chosen pair of elements $(x, y) \in H \times G$  equals $g$. We shall conclude this paper with some bounds for  the number of edges in $\Gamma_{H, G}^g$.


\section{Vertex degree and other properties}
For any vertex $x$ in $\Gamma_{H, G}^g$, we write $\deg(x)$ to denote the degree of $x$. In this section we first obtain computing formula for $\deg(x)$ in terms of  $|G|$, $|H|$ and  the orders of the centralizers of $x$.  

\begin{theorem}\label{deg_prop_1}
Let $x \in H$ be any vertex in  $\Gamma_{H, G}^g$.  
\begin{enumerate}
\item If $g=1$ then $\deg(x)= |G| - |C_G(x)|.$	
	
\item If $g \neq 1$ and $g^2 \neq 1$  then

$\deg(x) = \begin{cases}
|G| - |C_G(x)| - 1, \text{ if } x \text{ is conjugate to } xg \text{ or } xg^{-1}\\
|G| - 2|C_G(x)| - 1,\text{ if } x \text{ is conjugate to } xg \text{ and } xg^{-1}.   
\end{cases}$
	
\item If $g \neq 1$ and $g^2 = 1$  then $\deg(x) = |G| - |C_G(x)| - 1$, whenever $x$ is conjugate to $xg$.      
\end{enumerate}
\end{theorem}

\begin{proof}
\noindent (a) Let $g = 1$. Then $\deg(x)$ is the number of $y \in G$ such that $xy \ne yx$. Hence, $\deg(x)= |G| - |C_G(x)|$.

\noindent (b) Let  $g \neq 1$ and  $g^2 \neq 1$. Then $g \neq g^{-1}$.  Suppose that $x$  is conjugate to $xg$ or $xg^{-1}$  but not both. Without any loss we assume that $x$ is conjugate to $xg$. Then
there exits $y \in G$ such that $y^{-1}xy = xg$, that is  $[x, y] = x^{-1}y^{-1}xy = g$. Therefore, the set $S_g :=\{y \in G: y^{-1}xy = xg\}$ is non-empty. Also, for any $\alpha \in S_g$ we have $[x,\alpha]=g$ which gives that $\alpha$ is not adjacent to $x$. 
Thus, $\alpha \in G$ is not adjacent to $x$ if and only if $\alpha = x$ or $\alpha \in S_g$. Therefore, the number of  vertices  not adjacent to $x$ is equal to $|S_g| + 1$.

Let $y_1 \in S_g$ and $y_2 \in C_G(x)y_1$. Then $y_2 = uy_1$ for some $u \in C_G(x)$. We have
\[
y_2^{-1}xy_2 = y_1^{-1}u^{-1}x uy_1 = y_1^{-1}xy_1 = xg.
\]
Therefore, $y_2 \in S_g$ and so $C_G(x)y_1 \subseteq S_g$. Suppose that $y_3 \in S_g$. Then $y_1^{-1}xy_1 = y_3^{-1}xy_3$ which implies $y_3y_1^{-1} \in C_G(x)$. Therefore, $y_3 \in C_G(x)y_1$  and so $S_g \subseteq C_G(x)y_1$. Thus  $S_g = C_G(x)y_1$ and so $|S_g| = |C_G(x)|$. Hence, the number of  vertices  not adjacent to $x$ is equal to $|C_G(x)| + 1$
and so $\deg(x) = |G| - |C_G(x)| - 1$.



If $x$ is conjugate to $xg$ and $xg^{-1}$ then $S_g \cap S_{g^{-1}} = \emptyset$, where $S_{g^{-1}}=\{y \in G: y^{-1}xy = xg^{-1}\}$ and $|S_{g^{-1}}| = |C_G(x)|$. In this case, $\alpha \in G$ is not adjacent to $x$ if and only if $\alpha = x$ or $\alpha \in S_g \cup S_{g^{-1}}$. Therefore, the number of  vertices  not adjacent to $x$ is equal to $|S_g| + |S_{g^{-1}}|  + 1 = 2|C_G(x)| + 1$. 
Hence, $\deg(x) = |G| - 2|C_G(x)| - 1$.

\noindent (c) Let  $g \neq 1$ and  $g^2 = 1$. Then $g = g^{-1}$ and so $xg = xg^{-1}$. Now, if $x$ is conjugate to $xg$ then, as shown in the proof of part (b), we have
$\deg(x) = |G| - |C_G(x)| - 1$. 
\end{proof}

\begin{theorem}\label{deg_prop_2}
Let $x \in G \setminus H$ be any vertex in  $\Gamma_{H, G}^g$.  
\begin{enumerate}
\item If $g=1$ then $\deg(x)= |H| - |C_H(x)|.$
	
\item If $g \neq 1$ and $g^2 \neq 1$  then\\
$\deg(x) = \begin{cases}
|H| - |C_H(x)|, & \mbox{if $x$ is conjugate to $xg$ or $xg^{-1}$ for}\\
& \mbox{some element in $H$}. \\
|H| - 2|C_H(x)|, & \mbox{if $x$ is conjugate to $xg$ and $xg^{-1}$ for}\\
& \mbox{some element in $H$}.
\end{cases}$

\item If $g \neq 1$ and $g^2 = 1$  then $\deg(x) = |H| - |C_H(x)|$, whenever $x$ is conjugate to $xg$, for some element in H.
\end{enumerate}
\end{theorem}

\begin{proof}
Note that $x$ is not adjacent to any vertex in $G \setminus H$. Therefore, to prove the result, we consider only elements of $H$ and check whether they are adjacent/non-adjacent to $x$. 

\noindent (a) Let $g = 1$. Then $\deg(x)$ is the number of $y \in H$ such that $xy \ne yx$. Hence, $\deg(x)= |H| - |C_H(x)|$.

\noindent (b) Let  $g \neq 1$ and  $g^2 \neq 1$. Then $g \neq g^{-1}$. Suppose that $x$  is conjugate to $xg$ or $xg^{-1}$  (but not both) for some elements in $H$. Without any loss we assume that $x$ is conjugate to $xg$ for some elements in $H$. Then 
there exits $y \in H$ such that $y^{-1}xy = xg$, that is  $[x, y] = x^{-1}y^{-1}xy = g$. Therefore, the set $T_g :=\{y \in H : y^{-1}xy = xg\}$ is non-empty. Also, for any $\alpha \in T_g$ we have $[x,\alpha]=g$ which gives that $\alpha$ is not adjacent to $x$. 
Thus, $\alpha \in H$ is not adjacent to $x$ if and only if $\alpha \in S_g$. Therefore, the number of  vertices  not adjacent to $x$ is equal to $|T_g|$.

Let $y_1 \in T_g$ and $y_2 \in C_H(x)y_1$. Then $y_2 = uy_1$ for some $u \in C_H(x)$. We have
\[
y_2^{-1}xy_2 = y_1^{-1}u^{-1}x uy_1 = y_1^{-1}xy_1 = xg.
\]
Therefore, $y_2 \in T_g$ and so $C_H(x)y_1 \subseteq T_g$. Suppose that $y_3 \in T_g$. Then $y_1^{-1}xy_1 = y_3^{-1}xy_3$ which implies $y_3y_1^{-1} \in C_H(x)$. Therefore, $y_3 \in C_H(x)y_1$  and so $T_g \subseteq C_H(x)y_1$. Thus  $T_g = C_H(x)y_1$ and so $|T_g| = |C_H(x)|$. Hence, the number of  vertices  not adjacent to $x$ is equal to $|C_H(x)|$
and so $\deg(x) = |H| - |C_H(x)|$.


If $x$ is conjugate to $xg$ and $xg^{-1}$ for some elements of $H$ then $T_g \cap T_{g^{-1}} = \emptyset$, where $T_{g^{-1}}=\{y \in H : y^{-1}xy = xg^{-1}\}$ and $|T_{g^{-1}}|  = |C_H(x)|$. In this case, $\alpha \in H$ is not adjacent to $x$ if and only if $\alpha \in T_g \cup T_{g^{-1}}$. Therefore, the number of  vertices  not adjacent to $x$ is equal to $|T_g| + |T_{g^{-1}}|  = 2|C_H(x)|$. 
Hence, $\deg(x) = |H| - 2|C_H(x)|$.



\noindent (c) Let  $g \neq 1$ and  $g^2 = 1$. Then $g = g^{-1}$ and so $xg = xg^{-1}$. Now, if $x$ is conjugate to $xg$ for some elements in $H$ then, as shown in the proof of part (b), we have
$\deg(x) = |H| - |C_H(x)|$. 
%
%
\end{proof}

\noindent It is noteworthy that $g \notin K(H, G)$ if $x$ is not conjugate to $xg$ and $xg^{-1}$. Therefore, this case does not arise in 
 Theorem \ref{deg_prop_1} and Theorem \ref{deg_prop_2}. The degree of a vertex, in such case, is given by Observation \ref{deg_obs_1}. 

Now, we present some properties of $\Gamma^g_{H,G}$. The following lemmas are useful in this regard.
\begin{lemma}\label{deg_lemma_02}
If $g \ne 1$ and $H$ has an element of order $3$ then  $\Gamma^g_{H,G}$ is not triangle free.
\end{lemma}
\begin{proof}
Let $x$ be an element of $H$ having order $3$. Then, it is easy to see that the vertices $1, x$ and $x^{-1}$ forms a triangle in $\Gamma^g_{H,G}$. Hence, the lemma follows.
\end{proof}

\begin{lemma}\label{deg_lemma_2}
If $x \in Z(H,G)$ then  $\deg(x)= \begin{cases}
0, & \mbox{if $g=1$}\\
|G|-1, & \mbox{if $g \ne 1$.}\end{cases}$
\end{lemma}
\begin{proof}
By definition of $Z(H, G)$, it follows that $x \in H$ and $[x, y] = 1$ for all $y \in G$ and so $C_G(x) = G$.  Therefore, if $g = 1$ then by  Theorem \ref{deg_prop_1}(a) we have $\deg(x)= 0$.     If $g \ne 1$ then all the elements of $G$ except $x$ are adjacent to $x$. Therefore,  $\deg(x)= |G| - 1$.
\end{proof}
As a consequence of Lemma \ref{deg_lemma_2},  we have that the domination number of $\Gamma^g_{H,G}$ is one if $g \ne 1$ since $\{x\}$ is a dominating set for all $x \in Z(H, G)$. If $g$ is an element of $H$ having even order then it can be seen that $\{g\}$ is also a dominating set in $\Gamma^g_{H,G}$.  
If $g = 1$ then the domination number of $\Gamma^g_{H,G}$ is greater than or equal to $|Z(H, G)| + 1$. This lower bound is sharp because the  domination number of $\Gamma^{(1)}_{H,S_3}$ is $2 = |Z(H, S_3)| + 1$, where $H$ is any subgroup of $S_3$ of order $2$.   
If $g = 1$ then, by Lemma \ref{deg_lemma_2}, we also have that $\Gamma^g_{H,G}$ is disconnected. Hence, $\Gamma^1_{H,G}$ is not a tree, star graph and complete graph. In the following results we determine whether $\Gamma^g_{H,G}$ is a tree, star graph or complete graph if $g \ne 1$.


\begin{theorem}\label{not-tree}
Let $H$ be a  subgroup of $G$ and $|H| \ne 2$. If $g \ne 1$ then $\Gamma^g_{H,G}$ is not a tree.
\end{theorem}
\begin{proof}
Suppose for any  subgroup $H$ of $G$, $\Gamma^g_{H,G}$ is a tree.  Then there exits a vertex $x$ in $\Gamma^g_{H,G}$ of degree one.
Consider the following cases.

\noindent\textbf{Case 1:} $x \in H$

By Theorem  \ref{deg_prop_1} we have $\deg(x) = |G| - |C_G(x)| - 1 = 1$  or $\deg(x) = |G| - 2|C_G(x)| - 1 = 1$. That is,
\[
|G|-|C_G(x)| = 2  \text{ or } |G|-2|C_G(x)| = 2.
\]
Therefore, $|C_G(x)| = 2$ and $|G| = 4, 6$. Since $G$ is non-abelian and $|H| \ne 1, 2$, we must have  $G \cong S_3$ and $H = A_3$ or $S_3$. 
Therefore, by Lemma \ref{deg_lemma_02}, $\Gamma^g_{H,G}$ has a triangle which is a contradiction.


       
\noindent\textbf{Case 2:} $x \in G \setminus H$

By Theorem \ref{deg_prop_2} we have $\deg(x) = |H| - |C_H(x)| = 1$ or $\deg(x) = |H| - 2|C_H(x)| = 1$. Therefore, $|C_H(x)| = 1$ and $|H| = 2, 3$. However, $|H| \ne 2$ (by assumption). If $|H| = 3$ then, by Lemma \ref{deg_lemma_02}, $\Gamma^g_{H,G}$ has a triangle which is a contradiction.
%
%
Hence, the result follows.
\end{proof}
The proof of Theorem \ref{not-tree} also gives the following result.
\begin{theorem}
Let $H$ be a  subgroup of $G$ and $|H| \ne 2, 3$. If $g \ne 1$ then $\Gamma^g_{H,G}$ is not a lollipop. Further, if $|H| \ne 2, 3, 6$  then $\Gamma^g_{H,G}$ has no vertex of degree $1$.
\end{theorem}

As a consequence of Theorem \ref{not-tree} we have the following results.
\begin{corollary}
Let $H$ be a  subgroup of $G$ and $|H| \ne 2$. If $g \ne 1$ then $\Gamma^g_{H,G}$ is not a star graph.
\end{corollary}

\begin{corollary}
If $g \ne 1$ and $G$ is a group of odd order then $\Gamma^g_{H,G}$ is not a tree and hence not a star.
\end{corollary}

\begin{theorem}
If $g \ne 1$ then $\Gamma^g_{H,G}$ is  a star if and only if $G \cong S_3$ and $|H| = 2$.
\end{theorem}
\begin{proof}
By Lemma \ref{deg_lemma_2} we have  $\deg(1) = |G| - 1$. Suppose that  $\Gamma^g_{H,G}$ is a star graph. Then $\deg(x) = 1$ for all $1\ne x \in G$. Since $g \ne 1$ and $g \in K(H,G)$  we have $H \ne \{1\}$. 
Suppose that $1 \ne y \in H$.
%
If $g^2 =1$, then by Theorem \ref{deg_prop_1}, we have $1=\deg(y)=|G|-|C_G(y)|-1$ which gives $|G| = 4$, a contradiction since $G$ is non-abelian. If $g^2 \ne 1$, then by Theorem \ref{deg_prop_1}, we have $1=\deg(y)=|G|-|C_G(y)|-1$ or $|G|-2|C_G(y)|-1$ which gives   $|G| = 6$. Therefore, $G \cong S_3$, $g = (1 2 3), (1 3 2)$ and $H = \{(1), (1 2)\}, \{(1), (1 3)\}, \{(1), (2 3)\}$ or  $H = \{(1), (1 2 3), (1 3 2)\}$. If $|H| = 3$ then,  by Lemma \ref{deg_lemma_02}, $\Gamma^g_{H,S_3}$ is not a  star.
If $|H| = 2$ then it is easy to see that $\Gamma^g_{H,S_3}$ is a star. This completes the proof.
%
\end{proof}
\begin{theorem}
If $g \ne 1$ then $\Gamma^g_{H,G}$ is  not complete.
\end{theorem}
\begin{proof}
Suppose that $\Gamma^g_{H,G}$ is complete. Then $\deg(x) = |G| - 1$ for all $x \in G$. Since $g \ne 1$ and $g \in K(H,G)$  we have $H \ne \{1\}$. 
Suppose that $1 \ne y \in H$.
Then by Theorem \ref{deg_prop_1}, we have $|G| - 1 = \deg(y) = |G|-|C_G(y)|-1$ or $|G| - 1 = \deg(y) = |G| - 2|C_G(y)|-1$. Therefore, $|C_G(y)| = 0$, a contradiction. Hence, $\Gamma^g_{H,G}$ is not complete.
\end{proof}

\begin{theorem}
Let $H$ be a normal subgroup of $G$. If $g$ and $h$ are conjugate elements of  $G$ then $\Gamma^g_{H,G} \cong \Gamma^h_{H,G}$.
\end{theorem}
\begin{proof}
Let $h = g^x:=x^{-1}gx$ for some $x \in G$. Then for any two elements $a_1, a_2 \in G$, it is easy to see that 
\begin{equation}\label{eq-conj}
[a_1^x, a_2^x] = h \text{ or } h^{-1} \text{ if and only if } [a_1, a_2] = g \text{ or } g^{-1}.
\end{equation}
Consider the bijection $\phi: V(\Gamma^g_{H,G}) \to V(\Gamma^h_{H,G})$ given by $\phi(a) = a^x$ for all $x \in G$. We shall show that $\phi$ preserves adjacency.

Suppose that $a_1$ and $a_2$ are two elements of $V(\Gamma^h_{H,G})$. If $a_1$ and $a_2$ are not adjacent in $\Gamma^g_{H,G}$ then $[a_1, a_2] = g$ or $g^{-1}$. Therefore, by \eqref{eq-conj}, it follows that $\phi(a_1)$ and $\phi(a_2)$ are not adjacent in $\Gamma^h_{H,G}$. If $a_1$ and $a_2$ are adjacent then atleast one of $a_1$ and $a_2$ must belong to $H$ and  $[a_1, a_2] \ne g, g^{-1}$. Without any loss assume that $a_1 \in H$. Since $H$ is a normal subgroup of $G$ we have $\phi(a_1) \in H$. By \eqref{eq-conj}, we have $[\phi(a_1), \phi(a_2)] \ne h, h^{-1}$. Thus $\phi(a_1)$ and $\phi(a_2)$ are  adjacent in $\Gamma^h_{H,G}$. Hence, the result follows.
\end{proof}


A pair of isomorphisms $(\phi, \psi)$ is called a relative isoclinism between the pairs of groups $(H_1, G_1)$ and $(H_2, G_2)$, where $H_i \leq G_i$ for $i = 1, 2$, $\phi : \frac{G_1}{Z(H_1, G_1)} \to \frac{G_2}{Z(H_2, G_2)}$ and $\psi : [H_1, G_1] \to [H_2, G_2]$, if 
\[
\phi\left(\frac{H_1}{Z(H_1, G_1)}\right) = \frac{H_2}{Z(H_2, G_2)} \text{ and } \psi \circ a_{(H_1, G_1)} = a_{(H_2, G_2)} \circ (\phi \times \phi),
\]
where $a_{(H_i, G_i)} : \frac{H_i}{Z(H_i, G_i)} \times \frac{G_i}{Z(H_i, G_i)} \to [H_i, G_i]$ is given by 
\[
a_{(H_i, G_i)}((h_iZ(H_i, G_i), g_iZ(H_i, G_i))) = [h_i, g_i]
\]  
and $\phi \times \phi : \frac{H_1}{Z(H_1, G_1)} \times \frac{G_1}{Z(H_1, G_1)} \to  \frac{H_2}{Z(H_2, G_2)} \times \frac{G_2}{Z(H_2, G_2)}$ is given by 
\[
(\phi \times \phi)((h_1Z(H_1, G_1), g_1Z(H_1, G_1))) = (\phi(h_1Z(H_1, G_1)), \phi(g_1Z(H_1, G_1))).
\]
Thus for all $h_1 \in H_1$ and $g_1 \in G_1$ we must have $\psi ([h_1,g_1]) = [h_2,g_2]$, where $g_2 \in \phi(g_1Z(H_1,G_1))$ and $h_2 \in \phi(h_1Z(H_1,G_1))$.
 
The pairs $(H_1, G_1)$ and  $(H_2, G_2)$ are called relative isoclinic if there is a relative isoclinism between them. The concept of relative isoclinism between two pairs of groups was introduced in \cite{nath_yadav,SSK07,TE13}. This coincides with one of the fascinating concepts of Hall \cite{Hall} known as  isoclinism between two groups if $H_i = G_i$ for $i = 1, 2$. In \cite[Theorem 4.5]{TE13}, it was shown that $\Gamma^1_{H_1,G_1}$ is isomorphic to $\Gamma^1_{H_2,G_2}$ if $(H_1, G_1)$ and  $(H_2, G_2)$ are relative isoclinic satisfying certain conditions. Tolue et al. \cite[Theorem 2.16]{TEJ14},  also proved that $\Gamma_{G_1}^g$ is isomorphic to $\Gamma_{G_2}^{\psi(g)}$ if $G_1$ and $G_2$ are isoclinic such that $|Z(G_1)| = |Z(G_2)|$. We conclude this section with the following result which generalizes \cite[Theorem 2.16]{TEJ14}.
\begin{theorem}
Let $(\phi, \psi)$ be  a relative isoclinism between the pairs of groups $(H_1, G_1)$ and  $(H_2, G_2)$. If $|Z(H_1, G_1)| = |Z(H_2, G_2)|$ then $\Gamma_{H_1, G_1}^g$ is isomorphic to $\Gamma_{H_2, G_2}^{\psi(g)}$.
\end{theorem}
\begin{proof}
Since $\phi : \frac{G_1}{Z(H_1,G_1)} \to \frac{G_2}{Z(H_2,G_2)}$ is an isomorphism such that $\phi\left(\frac{H_1}{Z(H_1,G_1)}\right)$ $= \frac{H_2}{Z(H_2,G_2)}$. So we have $|\frac{H_1}{Z(H_1, G_1)}|  = |\frac{H_2}{Z(H_2, G_2)}|$ and $|\frac{G_1}{Z(H_1, G_1)}|  = |\frac{G_2}{Z(H_2, G_2)}|$. Let $\left|\frac{H_1}{Z(H_1,G_1)}\right| = |\frac{H_2}{Z(H_2, G_2)}| = m$ and $\left|\frac{G_1}{Z(H_1,G_1)}\right|= |\frac{G_2}{Z(H_2, G_2)}| =  n$. Given $|Z(H_1,G_1)|=|Z(H_2,G_2)|$, so $\exists$ a bijection $\theta:Z(H_1,G_1)\to Z(H_2,G_2)$. 
Let $\{h_1, h_2, \, \dots,\, h_m,$ $g_{m + 1}, \dots, g_n\}$  and  $\{h'_1, h'_2, \dots, h'_m,\, g'_{m + 1}, \, \dots, \, g'_n\}$ be two transversals of $\frac{G_1}{Z(H_1,G_1)}$ and $\frac{G_2}{Z(H_2,G_2)}$ respectively where  $\{h_1, h_2, \dots, h_m\}$ and $\{h'_1, h'_2,\dots, h'_m\}$ are transversals of $\frac{H_1}{Z(H_1, G_1)}$ and $\frac{H_2}{Z(H_2, G_2)}$ respectively. Let us define $\phi$ as $\phi(h_iZ(H_1,G_1)) = h'_iZ(H_2,G_2)$ and $\phi(g_jZ(H_1,G_1)) = g'_jZ(H_2,G_2)$ for $1 \leq i \leq m$ and $m + 1 \leq j \leq n$. 
	
	Let $\mu :G_1 \to G_2$ be a map such that $\mu(h_iz)=h'_i\theta(z)$, $\mu(g_jz)=g'_j\theta(z)$ for $z \in Z(H_1,G_1)$, $1 \leq i \leq m$ and $m + 1 \leq j \leq n$. Clearly $\mu$ is a bijection. Suppose two vertices $x$ and $y$ in $\Gamma_{H_1,G_1}^g$ are adjacent. Then $x \in H_1$ or $y \in H_1$ and  $[x, y] \neq g, g^{-1}$. Without any loss of generality, let us assume that $x \in H_1$. Then $x = h_iz_1$ for $1\leq i \leq m$ and $y = kz_2$ where  $z_1, z_2 \in Z(H_1, G_1)$, $k \in \{h_1, h_2, \dots, h_m, g_{m + 1}, \dots, g_n\}$. Therefore, for some $k' \in \{h'_1, \dots, h'_m, g'_{m + 1}, \dots, g'_n\}$, we have
\begin{align}\label{Iso-eq-001}
\psi([h_iz_1, kz_2])&= \psi([h_i, k]) 
= \psi \circ a_{(H_1, G_1)}\left((h_iZ(H_1, G_1), kZ(H_1, G_1))\right) \nonumber \\
& = a_{(H_2, G_2)} \circ (\phi \times \phi)\left((h_iZ(H_1, G_1), kZ(H_1, G_1))\right)\nonumber \\
& = a_{(H_2, G_2)}((h'_iZ(H_2, G_2), k'Z(H_2, G_2))) \nonumber\\
&= [h'_i, k'] 
= [h'_iz'_1, k'z'_2], 
\end{align}
where  $z'_1, z'_2 \in Z(H_2, G_2)$. 
Also,	
\begin{align*}
	&[h_iz_1, kz_2] \ne g, g^{-1}\\
	\Rightarrow \,& \psi([h_iz_1, kz_2]) \ne \psi(g), \psi(g^{-1})\\
	\Rightarrow \,& [h'_iz'_1, k'z'_2)] \ne \psi(g), \psi^{-1}(g) \text{ (using \eqref{Iso-eq-001})} \\
	\Rightarrow \,& [h'_i\theta(z_1), k'\theta(z_2)] \ne \psi(g), \psi^{-1}(g)\\ 
	\Rightarrow \,& [\mu(h_iz_1),\mu(kz_2)]  \ne \psi(g), \psi^{-1}(g)\\
	\Rightarrow \,& [\mu(x),\mu(y)]  \ne \psi(g), \psi^{-1}(g).
	\end{align*}
	Thus $\mu(x)$ is adjacent to  $\mu(y)$  in $\Gamma_{H_2, G_2}^{\psi(g)}$ since  $\mu(x) \in H_2$. Hence, the graphs $\Gamma_{H_1, G_1}^g$ and $\Gamma_{H_2, G_2}^{\psi(g)}$ are isomorphic under the map $\mu$.
\end{proof}
\section{Relation between $\Gamma^g_{H, G}$ and  ${\Pr}_g(H, G)$}
The commuting  probability of a finite group $G$ is the probability that a  randomly chosen pair of elements of it commute with each other. The popularity of this probability have been constantly increasing since its inception which is attributed to the works of Erd$\ddot{\rm o}$s and Tur$\acute {\rm a}$n \cite{pEpT68} published  in the year 1968. Many mathematicians worked on commuting probability and its generalizations and obtained valuable results towards classification of finite groups. Results related to this notion can be found in \cite{dnp2013} and the references listed there. 
Two most striking generalizations of commuting probability due to Pournaki et. al \cite{pournaki_sobhani} and Erfanian et. al \cite{Erfanian_Rezaei} are given by 
\[
 {\Pr}_g(G) := \frac{|\{(x, y) \in G^2 : [x, y] = g\}|}{|G|^2} 
\] 
and 
\[
 {\Pr}_1(H, G) := \frac{|\{(x, y) \in H \times G : [x, y] = 1\}|}{|H||G|}
\]  
respectively. Blending these notions, Nath together with Das and Yadav   \cite{das-nath-10,nath_yadav} considered the following generalization of commuting probability in their study
\[
 {\Pr}_g(H, G) := \frac{|\{(x, y) \in H \times G : [x, y] = g\}|}{|H||G|}.
\]  
In \cite{TE13}, Tolue and Erfanian  established some relations between ${\Pr}_1(H, G)$ and relative non-commuting graphs of finite groups. In \cite{TEJ14}, Tolue et al.  also established relations between $\Gamma^g_{G}$ and  ${\Pr}_g(G)$.  Their results stimulate us  to obtain relations between $\Gamma^g_{H, G}$ and ${\Pr}_g(H, G)$. 
In this section, we obtain the number of edges of $\Gamma^g_{H, G}$, denoted by $|E(\Gamma_{H, G}^g)|$, in terms of ${\Pr}_g(H, G)$.
Clearly, if $g \notin K(H,G)$ then from Observation \ref{deg_obs_1}, we have
\[
 2|E(\Gamma_{H, G}^g)| = 2|H||G| - |H|^2 - |H|.
\]
The following theorem gives expressions for $|E(\Gamma_{H, G}^g)|$, in terms of ${\Pr}_g(H, G)$ if $g \in K(H,G)$.

\begin{theorem}\label{rel-1}
Let $H$ be a non-trivial subgroup of a finite group $G$ and $g \in K(H,G)$.
\begin{enumerate}
\item If $g=1$ then
$
2|E(\Gamma_{H, G}^g)| = 2|H||G|(1 - {\Pr}_g(H, G)) - |H|^2(1-{\Pr}_g(H)).
$
\item If $g \neq 1$ and $g^2 = 1$  then
\[
2|E(\Gamma_{H, G}^g)| = \begin{cases}
2|H||G|(1 - {\Pr}_g(H, G)) \\
 ~~~~~~~~~~- |H|^2(1-{\Pr}_g(H)) - |H|, & \text{ if } g \in H \\
2|H||G|(1 - {\Pr}_g(H, G)) - |H|^2 - |H|, & \text{ if } g \in G \setminus H.
\end{cases}
\]
\item If $g \neq 1$ and $g^2 \neq 1$  then
\[
2|E(\Gamma_{H, G}^g)| = \begin{cases}
2|H||G|(1 - \underset{u = g, g^{-1}}{\sum}{\Pr}_u(H, G))\\
 ~~~~~- |H|^2(1-\underset{u = g, g^{-1}}{\sum}{\Pr}_u(H)) - |H|, & \text{ if } g \in H \\
2|H||G|(1 - \underset{u = g, g^{-1}}{\sum}{\Pr}_u(H, G))\\
 ~~~~~- |H|^2 - |H|, & \text{ if } g \in G \setminus H.
\end{cases}
\]
\end{enumerate}
\end{theorem}
\begin{proof}
Let $E_1 = \{(x, y) \in H \times G: x \ne y, [x,y] \neq g \text{ and } [x, y] \neq g^{-1}\}$ and $E_2 = \{(x, y) \in G \times H: x \ne y, [x,y] \neq g \text{ and } [x, y] \neq g^{-1}\}$. Clearly we have a bijection from $E_1$ to $E_2$ defined by $(x,y) \mapsto (y,x)$. So $|E_1| = |E_2|$. It is easy to see that $|E(\Gamma^g_{H, G})|$ is equal to half $|E_1 \cup E_2|$. Therefore, 
\begin{equation}\label{no_of_edges}
   2|E(\Gamma_{H, G}^g)| = 2|E_1| - |E_1 \cap E_2|, 
\end{equation} 
where $E_1 \cap E_2$ = $\{(x, y) \in H \times H: x \ne y, [x,y] \neq g \text{ and } [x, y] \neq g^{-1}\}$.

\noindent(a) If $g = 1$ then we have
\begin{align*}
|E_1| &= |\{(x, y) \in H \times G: [x,y] \neq 1 \}|\\
      &= |H||G| - |\{(x, y) \in H \times G: [x,y] = 1 \}|\\
      &= |H||G|(1-{\Pr}(H, G))  
\end{align*}
and
\begin{align*}
|E_1 \cap E_2| &= |\{(x, y) \in H \times H: [x,y] \neq 1 \}|\\
               &= |H|^2 - |\{(x, y) \in H \times H: [x,y] = 1 \}|\\
               &= |H|^2(1-{\Pr}(H)).    
\end{align*}
\noindent Hence, the result follows from \eqref{no_of_edges}.

\noindent(b) If $g \ne 1$  and  $g^2 = 1$ then  we have 
\begin{align*}
 |E_1| &= |\{(x, y) \in H \times G: x \ne y, [x,y] \neq g \}|\\
       &= |H||G| - |\{(x, y) \in H \times G: [x,y] = g \}| - |\{(x, y) \in H \times G: x = y \}|\\
       &= |H||G|(1 - {\Pr}_g(H, G)) - |H|.   
\end{align*}
Now, if $g \in H$ then
\begin{align*}
 |E_1 \cap &E_2| = |\{(x, y) \in H \times H: x \ne y, [x,y] \neq g \}|\\
                &= |H|^2 - |\{(x, y) \in H \times H: [x,y] = g \}| - |\{(x, y) \in H \times H: x = y \}|\\
                &= |H|^2(1-{\Pr}_g(H)) - |H|.   
\end{align*}
If $g \in G \setminus H$ then
\[
|E_1 \cap E_2| = |H|^2 - |H|.
\]
\noindent Hence, the result follows from \eqref{no_of_edges}.

\noindent (c) If g $\ne 1$  and  $g^2 \ne 1$ then we have
\begin{align*}
|E_1| &= |\{(x, y) \in H \times G: x \ne y, [x,y] \neq g \text{ and } [x,y] \neq g^{-1} \}|\\
      &= |H||G| - |\{(x, y) \in H \times G: [x,y] = g \}|\\
      &~~~~- |\{(x, y) \in H \times G: [x,y] = g^{-1} \}| - |\{(x, y) \in H \times H: x = y \}|\\
      &= |H||G|(1 - \underset{u = g, g^{-1}}{\sum}{\Pr}_u(H, G)) - |H|.
\end{align*} 
 Now, if $g \in H$ then
\begin{align*}
  |E_1 \cap E_2| &= |\{(x, y) \in H \times H: x \ne y, [x,y] \neq g \text{ and } [x,y] \neq g^{-1}  \}|\\
&= |H|^2 - |\{(x, y) \in H \times H: [x,y] = g \}| \\
&~~~~ -|\{(x, y) \in H \times H: [x,y] = g^{-1} \}|
    - |\{(x, y) \in H \times H: x = y \}|\\
&= |H|^2(1 - \underset{u = g, g^{-1}}{\sum}{\Pr}_u(H)) - |H|.  
\end{align*}
If $g \in G \setminus H$ then
\[
|E_1 \cap E_2| = |H|^2 - |H|.
\]
\noindent Hence, the result follows from \eqref{no_of_edges}.
\end{proof}
If $H$ is an abelian  group then we have 
\[
{\Pr}_g(H) = \begin{cases}
		1, &\mbox{if } g = 1 \\
		0, &\mbox{if } g \neq 1.
		\end{cases}
\]
Using these values in  Theorem \ref{rel-1} we get the following corollary.
\begin{corollary}\label{rel-cor-01}
Let $H$ be an abelian non-trivial subgroup of a finite group $G$ and $g \in K(H,G)$.
\begin{enumerate}
\item If $g=1$ then
$
|E(\Gamma_{H, G}^g)| = |H||G|(1 - {\Pr}_g(H, G)).
$
\item If $g \neq 1$ and $g^2 = 1$  then
\[
2|E(\Gamma_{H, G}^g)| = 2|H||G|(1 - {\Pr}_g(H, G)) - |H|^2 - |H|.
\]
\item If $g \neq 1$ and $g^2 \neq 1$  then
\[
2|E(\Gamma_{H, G}^g)| = 2|H||G|(1 - \underset{u = g, g^{-1}}{\sum}{\Pr}_u(H, G)) - |H|^2 - |H|.
\]
\end{enumerate}
\end{corollary}

\begin{proposition}\label{com-prop-1}
Let $H$ be a subgroup of a finite group $G$ and $p$ be the smallest prime dividing $|G|$. Let $|[H, G]| = p$ and $g \in K(H, G)$. 
\begin{enumerate}
\item If $g = 1$ then
\[
2p|E(\Gamma_{H, G}^g)| = (p-1)[2|G|(|H| - |Z(H,G)|) - |H|(|H| - |Z(H)|)].
\]
\item If $g \neq 1$ and $g^2 = 1$ then
\[
2p|E(\Gamma_{H, G}^g)| = \begin{cases}
2|G|((p-1)|H| + |Z(H,G)|) &\\
~~~~- |H|((p-1)|H| + |Z(H)| + p), & \text{ if } g \in H \\
2|G|((p-1)|H| + |Z(H,G)|) &\\
~~~~- p|H|(|H| + 1), & \text{ if } g \in G \setminus H.
\end{cases}
\]
\item If $g \neq 1$ and $g^2 \neq 1$ then
\[
2p|E(\Gamma_{H, G}^g)| = \begin{cases}
2|G|((p-2)|H| + 2|Z(H,G)|) &\\
~~~- |H|((p-2)|H| + 2|Z(H)| + p), & \text{ if } g \in H \\
2|G|((p-2)|H| + 2|Z(H,G)|) &\\
~~~- p|H|(|H| + 1), & \text{ if } g \in G \setminus H.
\end{cases}
\]
\end{enumerate}
\end{proposition}
\begin{proof}
By \cite[Lemma 3]{nath_yadav}, we have 
\[
{\Pr}_g(H, G) = \begin{cases}
\frac{1}{p}\left(1 + \frac{p - 1}{|H:Z(H,G)|}\right), &\mbox{if } g = 1 \\
\frac{1}{p}\left(1 - \frac{1}{|H:Z(H,G)|}\right), &\mbox{if } g \neq 1.
\end{cases}
\]
Hence, the result follows from Theorem \ref{rel-1}.
\end{proof}
It is worth mentioning that, in view of \cite[Theorem B]{nath_yadav}, the conclusion of Proposition \ref{com-prop-1} also holds if $H$ is a subgroup of a finite nilpotent group $G$ such that $|[H, G]| = p$, where $p$ is a prime not necessarily the smallest one dividing $|G|$. We also have the following corollary.

\begin{corollary}\label{com-cor-2}
Let $H$ be an abelian subgroup of a finite nilpotent group $G$.  Let $|[H, G]| = p$, a prime (not necessarily the smallest one dividing $|G|$) and $g \in K(H, G)$. 
\begin{enumerate}
\item If $g = 1$ then
$
p|E(\Gamma_{H, G}^g)| = (p-1)|G|(|H| - |Z(H,G)|).
$
\item If $g \neq 1$ and $g^2 = 1$ then
\[
2p|E(\Gamma_{H, G}^g)| = 2|G|((p-1)|H| + |Z(H,G)|) - p|H|(|H| + 1).
\]
\item If $g \neq 1$ and $g^2 \neq 1$ then
\[
2p|E(\Gamma_{H, G}^g)| = 2|G|((p-2)|H| + 2|Z(H,G)|) - p|H|(|H| + 1).
\]
\end{enumerate}
\end{corollary}
In \cite[Proposition 2.14]{TEJ14}, Toule et al. obtained a relation between $|E(\Gamma_{G}^g)|$ and ${\Pr}_g(G)$. It is noteworthy that their result can also be obtained from the next proposition considering $H = G$.
\begin{proposition}\label{rel-3}
Let $H$ be a non-trivial normal subgroup of a finite group $G$ and $g \in K(H,G)$.
\begin{enumerate}
\item If $g=1$ then
$
2|E(\Gamma_{H, G}^1)| = (2|G| - |H|)(|H| - k(H)),
$
where $k(H)$ is the number of conjugacy classes in $H$.
\item If $g \neq 1$ and $g^2 = 1$  then
\[
2|E(\Gamma_{H, G}^g)| = 2|H||G|(1 - {\Pr}_g(H, G)) 
- |H|^2(1-{\Pr}_g(H)) - |H|.
\]
\item If $g \neq 1$ and $g^2 \neq 1$  then
\begin{align*}
2|&E(\Gamma_{H, G}^g)|
 = 2|H||G|(1 - 2{\Pr}_g(H, G))\
- |H|^2(1-2{\Pr}_g(H)) - |H|.
\end{align*}
\end{enumerate}
\end{proposition}
\begin{proof}
If $g = 1$ then by \cite[Corollary 2.4]{das-nath-10} we have  
\[
{\Pr}_1(H, G)) = {\Pr}_1(H)) = \frac{k(H)}{|H|},
\]
where $k(H)$ is the number of conjugacy classes in $H$. Hence, part (a) follows from Theorem \ref{rel-1}. Parts (b) and (c) also follow from Theorem \ref{rel-1} noting that the case $g \in G \setminus H$ does not arise (since $g \in H$ if $H$ is normal) and ${\Pr}_g(H, G)) = {\Pr}_{g^{-1}}(H, G))$ (as shown in \cite[Proposition 2.1]{das-nath-10}).
\end{proof}
If $H$ is a normal subgroup of $G$ then by \cite[Equation (6)]{das-nath-10} we have 
\[
{\Pr}_g(H, G)) = \frac{1}{|G|} \underset{\phi \in \Irr(G)}{\sum}\langle\phi_H, \phi_H\rangle\frac{\phi(g)}{\phi(1)},
\] 
where $\Irr(G)$ is the set of all irreducible characters of $G$, $\phi_H$ is the restriction of $\phi \in \Irr(G)$ on $H$ and  $\langle , \rangle$ represents inner product of class functions of $G$. In view of the above formula for ${\Pr}_g(H, G))$  and  Proposition  \ref{rel-3} we get the following character theoretic formula for $|E(\Gamma_{H, G}^g)|$.
\begin{corollary}\label{rel-4}
Let $H$ be a non-trivial normal subgroup of a finite group $G$ and $g \in K(H,G)$.
\begin{enumerate}
\item If $g=1$ then
$
2|E(\Gamma_{H, G}^1)| = (2|G| - |H|)(|H| - |\Irr(H)|).
$
\item If $g \neq 1$ and $g^2 = 1$  then
\begin{align*}
2|E(\Gamma_{H, G}^g)| = & 2|H|\left(|G|  -  \underset{\phi \in \Irr(G)}{\sum}\langle\phi_H, \phi_H\rangle\frac{\phi(g)}{\phi(1)}\right)\\ 
& ~~~~~~~~~~~~~~~~~~~~~~~~~~~~~~~ - |H|\left(|H| - \underset{\phi \in \Irr(H)}{\sum}\frac{\phi(g)}{\phi(1)}\right) - |H|.
\end{align*}
\item If $g \neq 1$ and $g^2 \neq 1$  then
\begin{align*}
2|E(\Gamma_{H, G}^g)|
 = & 2|H|\left(|G| - 2\underset{\phi \in \Irr(G)}{\sum}\langle\phi_H, \phi_H\rangle\frac{\phi(g)}{\phi(1)}\right)\\
&~~~~~~~~~~~~~~~~~~~~~~~~~~~~~~~- |H|\left(|H|-2\underset{\phi \in \Irr(H)}{\sum}\frac{\phi(g)}{\phi(1)}\right) - |H|.
\end{align*}
\end{enumerate}
\end{corollary}
 
 \begin{corollary}\label{rel-5}
Let $G$ be a finite group and $g \in K(G)$.
\begin{enumerate}
\item If $g=1$ then $2|E(\Gamma_{G}^1)| = |G|(|G| - |\Irr(G)|)$.
\item If $g \neq 1$ and $g^2 = 1$  then
$2|E(\Gamma_{G}^g)| = |G|\left(|G| - 1  -  \underset{\phi \in \Irr(G)}{\sum}\frac{\phi(g)}{\phi(1)}\right).$
\item If $g \neq 1$ and $g^2 \neq 1$  then
$2|E(\Gamma_{G}^g)|
 = |G|\left(|G| - 1 - 2\underset{\phi \in \Irr(G)}{\sum}\frac{\phi(g)}{\phi(1)}\right).$
\end{enumerate}
\end{corollary}
\section{Bounds for $|E(\Gamma_{G}^g)|$}
In \cite[Section 3]{TE13}, Tolue and Erfanian obtained bounds for the number of edges in $\Gamma_{H, G}$. In this section we obtain some bounds for the number of edges in $\Gamma_{H, G}^g$.  By Theorem \ref{rel-1}, we have
\begin{equation}\label{Thm3.1-eq1}
2|E(\Gamma_{H, G}^g)| + |H|^2 + |H| = \begin{cases}
2|H||G|(1 - {\Pr}_g(H, G)) \\
 ~~~~~~~~~~+ |H|^2{\Pr}_g(H), & \text{ if } g \in H \\
2|H||G|(1 - {\Pr}_g(H, G)), & \text{ if } g \in G \setminus H,
\end{cases}
\end{equation}
if $g \neq 1$ but $g^2=1$ and
\begin{equation}\label{Thm3.1-eq2}
2|E(\Gamma_{H, G}^g)| + |H|^2+ |H| = \begin{cases}
2|H||G|(1 - \underset{u = g, g^{-1}}{\sum}{\Pr}_u(H, G))\\
 ~~~~~+|H|^2\underset{u = g, g^{-1}}{\sum}{\Pr}_u(H), & \text{ if } g \in H \\
2|H||G|(1 - \underset{u = g, g^{-1}}{\sum}{\Pr}_u(H, G)), & \text{ if } g \in G \setminus H,
\end{cases}
\end{equation}
if $g \neq 1$ and $g^2 \neq 1$.
\begin{proposition}
Let $H$ be a subgroup of a finite group $G$ and  $g \neq 1$. 
\begin{enumerate}
\item If $g^2 = 1$  then
\[
|E(\Gamma_{H, G}^g)| \geq \begin{cases}
\frac{|H|(|G| - 1) + |G||Z(H, G)| + 3|Z(H)|^2 -|H|^2}{2}, & \text{ if } g \in H \\
\frac{|H|(|G| - 1) + |G||Z(H, G)| -|H|^2}{2}, & \text{ if } g \in G \setminus H.
\end{cases}
\]
\item If $g^2 \neq 1$ then
\[
|E(\Gamma_{H, G}^g)| \geq \begin{cases}
\frac{2|G||Z(H, G)| + 6|Z(H)|^2 - |H|^2 - |H|}{2},  & \text{ if }  g \in  H\\
\frac{2|G||Z(H, G)| - |H|^2 - |H|}{2}, & \text{ if } g \in G \setminus H.
\end{cases}
\]
\end{enumerate}
\end{proposition}
\begin{proof}
By \cite[Proposition 3.3]{das-nath-10}, we have  
\begin{equation}\label{new-up-bd}
1 - {\Pr}_g(H, G) \geq \frac{|H| + |Z(H, G)|}{2|H|}
\text{ and }
%
1 - \underset{u = g, g^{-1}}{\sum}{\Pr}_u(H, G)  \geq \frac{|Z(H, G)|}{|H|}.
\end{equation}
Again, by \cite[Proposition 3.1 (iii)]{das-nath-10}, we  have
\begin{equation}\label{new-up-bd1}
{\Pr}_g(H) \geq \frac{3|Z(H)|^2}{|H|^2}.
\end{equation}
(a) We have  $g^2 = 1$.
Therefore,  if  $g \in H$ then, using \eqref{Thm3.1-eq1}, \eqref{new-up-bd} and \eqref{new-up-bd1}, we get
\begin{equation}\label{new-sup-up-bd1}
2|E(\Gamma_{H, G}^g)| + |H|^2 + |H| \geq  2|H||G|\left(\frac{|H| + |Z(H, G)|}{2|H|}\right) + |H|^2\left(\frac{3|Z(H)|^2}{|H|^2}\right).
\end{equation}
If $g \in G \setminus H$ then,  using \eqref{Thm3.1-eq1} and \eqref{new-up-bd}, we get
\begin{equation}\label{new-up-bd-001}
2|E(\Gamma_{H, G}^g)| + |H|^2 + |H| > 2|H||G|\left(\frac{|H| + |Z(H, G)|}{2|H|}\right).
\end{equation}
 Hence, the result follows from \eqref{new-sup-up-bd1} and \eqref{new-up-bd-001}.

\noindent(b) We have $g^2 \ne 1$.
Therefore, if $g \in H$ then, using \eqref{Thm3.1-eq2}, \eqref{new-up-bd} and \eqref{new-up-bd1}, we get
\begin{equation}\label{new-up-bd-02}
2|E(\Gamma_{H, G}^g)| + |H|^2+ |H| \geq  2|H||G|\left(\frac{|Z(H, G)|}{|H|}\right) + |H|^2\left(\frac{6|Z(H)|^2}{|H|^2}\right).
\end{equation}
\noindent If $g \in G \setminus H$ then, using \eqref{Thm3.1-eq2} and \eqref{new-up-bd}, we have
\begin{equation}\label{new-up-bd-002}
2|E(\Gamma_{H, G}^g)| + |H|^2+ |H| \geq 2|H||G|\left(\frac{|Z(H, G)|}{|H|}\right). 
\end{equation}
\noindent Hence, the result follows from  \eqref{new-up-bd-02} and \eqref{new-up-bd-002}.
\end{proof}
\begin{proposition}
Let $H$ be a subgroup of a finite group $G$ and  $g \neq 1$. 
\begin{enumerate}
\item If $g^2 = 1$  then
\[
|E(\Gamma_{H, G}^g)| \leq \begin{cases}
\frac{4|H||G|-8|Z(H,G)||Z(G, H)|-|H|^2 -|H|(|Z(H)| + 2)}{4}, & \text{ if } g \in H \\
\frac{2|H||G|-4|Z(H,G)||Z(G, H)|-|H|^2-|H|}{2}, & \text{ if } g \in G \setminus H.
\end{cases}
\]
\item If $g^2 \neq 1$  then
\[
|E(\Gamma_{H, G}^g)| \leq \begin{cases}
\frac{2|H||G|-8|Z(H,G)||Z(G, H)|- |H|(|Z(H)| + 1)}{2}, & \text{ if }  g \in  H \\
\frac{2|H||G|-8|Z(H,G)||Z(G, H)|-|H|^2-|H|}{2}, & \text{ if }  g \in G \setminus H.
\end{cases}
\]
\end{enumerate}
\end{proposition}
\begin{proof}
By \cite[Proposition 3.1 (ii)]{das-nath-10},  we have
\begin{equation}\label{up--bd}
1-{\Pr}_g(H, G) \leq \frac{|H||G| - 2|Z(H,G)||Z(G, H)|}{|H||G|}
\end{equation}
and
\begin{equation}\label{up--bd-010}
1-\underset{u = g,g^{-1}}{\sum}{\Pr}_u(H, G)) \leq \frac{|H||G| - 4|Z(H,G)||Z(G, H)|}{|H||G|}.
\end{equation}
Also, by \cite[Proposition 3.3]{das-nath-10},  we have
\begin{equation}\label{up--bd1}
{\Pr}_g(H) \leq \frac{|H|- |Z(H)|}{2|H|}.
\end{equation}
\noindent (a) We have $g^2=1$.
Therefore, if $g \in H$ then, using \eqref{Thm3.1-eq1},  \eqref{up--bd} and \eqref{up--bd1}, we get
\begin{align}\label{up--bd--01}
2|E(\Gamma_{H, G}^g)| + |H|^2 + |H| \leq & \, 2|H||G|\left(\frac{|H||G| - 2|Z(H,G)||Z(G, H)|}{|H||G|}\right) \nonumber\\
& + |H|^2\left(\frac{|H|- |Z(H)|}{2|H|}\right).
\end{align}
\noindent If  $g \in G \setminus H$ then, using   \eqref{Thm3.1-eq1} and \eqref{up--bd}, we get
\begin{equation}\label{up--bd--001}
2|E(\Gamma_{H, G}^g)| + |H|^2 + |H| \leq 2|H||G|\left(\frac{|H||G| - 2|Z(H,G)||Z(G, H)|}{|H||G|}\right).
\end{equation}
\noindent Hence, the result follows from \eqref{up--bd--01} and \eqref{up--bd--001}.

\noindent (b) We have $g^2 \ne 1$.
Therefore, if $g \in H$ then, using \eqref{Thm3.1-eq2}, \eqref{up--bd-010} and \eqref{up--bd1}, we get
\begin{align}\label{up--bd--02}
2|E(\Gamma_{H, G}^g)| + |H|^2 + |H| \leq & \, 2|H||G|\left(\frac{|H||G| - 4|Z(H,G)||Z(G, H)|}{|H||G|}\right)\nonumber \\ 
& + |H|^2\left(\frac{|H|- |Z(H)|}{|H|}\right).
\end{align}
\noindent If $g \in G \setminus H$ then,  using \eqref{Thm3.1-eq2} and \eqref{up--bd-010}, we get
\begin{equation}\label{up--bd--002}
2|E(\Gamma_{H, G}^g)| + |H|^2 + |H| \leq 2|H||G|\left(\frac{|H||G| - 4|Z(H,G)||Z(G, H)|}{|H||G|}\right).
\end{equation}
\noindent Hence, the result follows from  \eqref{up--bd--02}  and \eqref{up--bd--002}.
\end{proof}

\begin{proposition}
Let $p$ be the smallest prime dividing $|G|$ and $g \neq 1$. Then for any subgroup $H$ of $G$ we have the following bounds for $|E(\Gamma_{H, G}^g)|$.
\begin{enumerate}
\item If $g^2=1$ then
\[
|E(\Gamma_{H, G}^g)| \geq \begin{cases}
\frac{2(p-1)|H||G|+ 2|Z(H,G)||G| - p|H|^2 +3p|Z(H)|^2 - p|H|}{2p}, & \text{ if } g \in H\\
\frac{2(p-1)|H||G|+ 2|Z(H,G)||G| - p|H|^2 - p|H|}{2p}, & \text{ if } g \in G \setminus H.
\end{cases}
\]
\item If $g^2 \ne 1$ then
\[
|E(\Gamma_{H, G}^g)| \geq \begin{cases}
\frac{2(p-2)|H||G|+ 4|Z(H,G)||G| - p|H|^2 +6p|Z(H)|^2 - p|H|}{2p}, & \text{ if } g \in H\\
\frac{2(p-2)|H||G|+ 4|Z(H,G)||G| - p|H|^2 - p|H|}{2p}, & \text{ if } g \in G \setminus H.
\end{cases}
\]
\end{enumerate}
\end{proposition}
\begin{proof}
By \cite[Proposition 3.3]{das-nath-10}, we have  
\begin{equation}\label{up-bd}
1 - {\Pr}_g(H, G) \geq \frac{(p-1)|H|+|Z(H,G)|}{p|H|}
\end{equation}
and 
\begin{equation}\label{up-bd-01}
1 - \underset{u = g, g^{-1}}{\sum}{\Pr}_u(H, G) \geq \frac{(p-2)|H|+2|Z(H,G)|}{p|H|}.
\end{equation}
(a) We have  $g^2 = 1$.
Therefore,  if  $g \in H$ then, using \eqref{Thm3.1-eq1}, \eqref{up-bd} and \eqref{new-up-bd1}, we get
\begin{equation}\label{sup-up-bd1}
2|E(\Gamma_{H, G}^g)| + |H|^2 + |H| \geq  2|H||G|\left(\frac{(p-1)|H|+|Z(H,G)|}{p|H|}\right) + |H|^2\left(\frac{3|Z(H)|^2}{|H|^2}\right).
\end{equation}
If $g \in G \setminus H$ then,  using \eqref{Thm3.1-eq1} and \eqref{up-bd}, we get
\begin{equation}\label{up-bd-001}
2|E(\Gamma_{H, G}^g)| + |H|^2 + |H| \geq 2|H||G|\left(\frac{(p-1)|H|+|Z(H,G)|}{p|H|}\right).
\end{equation}
 Hence, the result follows from \eqref{sup-up-bd1} and \eqref{up-bd-001}.

\noindent(b) We have $g^2 \ne 1$.
Therefore, if $g \in H$ then, using \eqref{Thm3.1-eq2}, \eqref{up-bd-01} and \eqref{new-up-bd1}, we get
\begin{equation}\label{up-bd-02}
2|E(\Gamma_{H, G}^g)| + |H|^2+ |H| \geq  2|H||G|\left(\frac{(p-2)|H|+2|Z(H,G)|}{p|H|}\right) + |H|^2\left(\frac{6|Z(H)|^2}{|H|^2}\right).
\end{equation}
\noindent If $g \in G \setminus H$ then, using \eqref{Thm3.1-eq2} and \eqref{up-bd-01}, we have
\begin{equation}\label{up-bd-002}
2|E(\Gamma_{H, G}^g)| + |H|^2+ |H| \geq  2|H||G|\left(\frac{(p-2)|H|+2|Z(H,G)|}{p|H|}\right). 
\end{equation}
\noindent Hence, the result follows from  \eqref{up-bd-02}  and \eqref{up-bd-002}.
\end{proof}

\begin{proposition}
Let $p$ be the smallest prime dividing $|G|$ and $g \neq 1$. Then for any subgroup $H$ of $G$ we have the following bounds for $|E(\Gamma_{H, G}^g)|$.
\begin{enumerate}
\item If $g^2=1$ then
\[
|E(\Gamma_{H, G}^g)| \leq \begin{cases}
\frac{2p|H||G|-4p|Z(H,G)||Z(G, H)|-(p-1)|H|^2-|H||Z(H)|-p|H|}{2p},  \!\text{ if } g \in H\\
\frac{2|H||G|-4|Z(H,G)||Z(G, H)|-|H|^2-|H|}{2},  ~~~~~~~~~~~~~~\text{ if } g \in G \setminus H.
\end{cases}
\]
\item If $g^2 \neq 1$
\[
|E(\Gamma_{H, G}^g)|\!\leq \begin{cases}
\frac{2p|H||G|-8p|Z(H,G)||Z(G, H)|-(p-2)|H|^2-2|H||Z(H)|-p|H|}{2p}, \!\!\text{ if } g\in H\\
\frac{2|H||G|-8|Z(H,G)||Z(G, H)|-|H|^2-|H|}{2},  ~~~~~~~~~~~~~~~\text{ if } g \in G \setminus H.
\end{cases}
\]
\end{enumerate}
\end{proposition}
\begin{proof}
By \cite[Proposition 3.3]{das-nath-10},  we get  
\begin{equation}\label{up---bd}
{\Pr}_g(H) \leq \frac{|H|-|Z(H)|}{p|H|}.
\end{equation}

\noindent (a) We have $g^2=1$.
Therefore, if $g \in H$ then, using \eqref{Thm3.1-eq1},  \eqref{up--bd} and \eqref{up---bd}, we get
\begin{align}\label{lw--bd--01}
2|E(\Gamma_{H, G}^g)| + |H|^2 + |H| \leq & \, 2|H||G|\left(\frac{|H||G| - 2|Z(H,G)||Z(G, H)|}{|H||G|}\right)\nonumber\\
& + |H|^2\left(\frac{|H|-|Z(H)|}{p|H|}\right). 
\end{align}
\noindent If $g \in G \setminus H$ then, using   \eqref{Thm3.1-eq1} and \eqref{up--bd}, we get
\begin{equation}\label{lw--bd--001}
2|E(\Gamma_{H, G}^g)| + |H|^2 + |H| \leq 2|H||G|\left(\frac{|H||G| - 2|Z(H,G)||Z(G, H)|}{|H||G|}\right). 
\end{equation}
\noindent Hence, the result follows from \eqref{lw--bd--01} and \eqref{lw--bd--001}.

\noindent (b) We have $g^2 \ne 1$.
Therefore, if $g \in H$ then, using \eqref{Thm3.1-eq2}, \eqref{up--bd-010} and \eqref{up---bd}, we get
\begin{align}\label{lw--bd--02}
2|E(\Gamma_{H, G}^g)| + |H|^2+ |H| \leq & \, 2|H||G|\left(\frac{|H||G| - 4|Z(H,G)||Z(G, H)|}{|H||G|}\right) \nonumber\\
& + 2|H|^2\left(\frac{|H|-|Z(H)|}{p|H|}\right).
\end{align}
\noindent If $g \in G \setminus H$ then,  using \eqref{Thm3.1-eq2} and \eqref{up--bd-010}, we get
\begin{equation}\label{lw--bd--002}
2|E(\Gamma_{H, G}^g)| + |H|^2+ |H| \leq 2|H||G|\left(\frac{|H||G| - 4|Z(H,G)||Z(G, H)|}{|H||G|}\right). 
\end{equation}
\noindent Hence, the result follows from  \eqref{lw--bd--02}  and \eqref{lw--bd--002}.
\end{proof}

Note that several other bounds for $|E(\Gamma_{H, G}^g)|$ can be obtained using different combinations of the bounds for  ${\Pr}_g(H, G)$ and ${\Pr}_g(H)$.
We conclude this paper with the following bounds for $|E(\Gamma_{G}^g)|$ which are obtained by putting $H = G$ in the above propositions.
\begin{corollary}
Let $H$ be a subgroup of a finite group $G$ and  $g \neq 1$. 
\begin{enumerate}
\item If $g^2 = 1$  then
$
\frac{|G||Z(G)| + 3|Z(G)|^2 -|G|}{2} \leq
|E(\Gamma_{G}^g)|  \leq 
\frac{3|G|^2-8|Z(G)|^2 -|G|(|Z(G)| + 2)}{4} 
$.
\item If $g^2 \neq 1$  then
\begin{center}
$
\frac{2|G||Z(G)| + 6|Z(G)|^2 - |G|^2 - |G|}{2} \leq
|E(\Gamma_{G}^g)| \leq 
\frac{2|G|^2-8|Z(G)|^2- |G|(|Z(G)| + 1)}{2}
$.
\end{center}
\end{enumerate}
\end{corollary}
\begin{corollary}
Let $p$ be the smallest prime dividing $|G|$ and $g \neq 1$. 
\begin{enumerate}
\item If $g^2=1$ then
\begin{align*}
&\frac{(p-2)|G|^2+ 2|Z(G)||G| +3p|Z(G)|^2 - p|G|}{2p} \leq  |E(\Gamma_{G}^g)| \\ 
&~~~~~~~~~~~~~~~~~~~~~~~~~~~~~~~~\leq \frac{(p +1)|G|^2-4p|Z(G)|^2 -|G||Z(G)|-p|G|}{2p}.
\end{align*}
\item If $g^2 \ne 1$ then
\begin{align*}
&\frac{(p-4)|G|^2+ 4|Z(G)||G| +6p|Z(G)|^2 - p|G|}{2p} \leq |E(\Gamma_{G}^g)| \\ 
&~~~~~~~~~~~~~~~~~~~~~~~~~~~~~~~~\leq \frac{(p +2)|G|^2-8p|Z(G)|^2-2|G||Z(G)|-p|G|}{2p}. 
\end{align*}
\end{enumerate}
\end{corollary}

\section*{Acknowledgment}
The first author would like to thank  DST for the INSPIRE Fellowship.

\end{document}